\documentclass[12pt]{amsart}
\usepackage{fullpage}
\usepackage{amsthm}
\usepackage{amsmath}
\usepackage{amssymb}
\usepackage{booktabs}
\usepackage[ruled,vlined,linesnumbered]{algorithm2e}
\usepackage{comment}
\usepackage{float}
\usepackage{array}
\usepackage{placeins}
\usepackage{enumitem}
\usepackage{hyperref}

\makeatletter
\@namedef{subjclassname@2020}{%
  \textup{2020} Mathematics Subject Classification}
\makeatother

\newtheorem{thm}{Theorem}[section]

\newtheorem{lemma}[thm]{Lemma}
\newtheorem{cor}[thm]{Corollary}
\newtheorem{remark}{Remark}[section]
\newtheorem{question}{Question}[section]

\numberwithin{equation}{section}

\newcommand{\F}{\mathbb{F}_q}
\newcommand{\Fp}{\mathbb{F}_p}

\title[Non-square non-primitive pairs] {Consecutive non-square non-primitive \\pairs in a Finite Field}

\author{Stephen D. Cohen}

\date{} 
\begin{document}

\begin{abstract}
Let $q$ be an odd prime power and write
\[
\theta_q := \frac{\phi(q-1)}{q-1}.
\]
If $\theta_q < \tfrac{1}{3}$, or if $\theta_q = \tfrac{1}{3}$ and 
$q \notin \{7,13,19,25,37\}$, then the finite field $\F$ contains a pair of consecutive elements that are both non-square and non-primitive. This extends a result of Jarso and Trudgian for prime fields $\Fp$, where the same conclusion was obtained under the stronger condition $\theta_p \le \tfrac{1}{4}$.

More generally, let $\ell$ be the least odd prime divisor of $q-1$. If $\theta_q \le \tfrac{1}{3}$, then $\F$ contains a pair of consecutive elements that are non-squares and $\ell$th powers, with the sole exceptions 
$q \in \{7,13,19,25,37,43\}$.
\end{abstract}

\maketitle

\section{Introduction}

Let $\F$ denote the finite field of order $q$. In a series of papers published in 1985, the author \cite{Coh1985a, Coh1985b, Coh1985} proved that, provided $q\neq 2,3,7$, the field $\F$ contains a consecutive pair of primitive elements. Many related existence results have since been obtained.

A natural complementary question is whether there always exists a consecutive pair of \emph{non-primitive} elements. When $q$ is odd, this question requires some qualification: exactly half of the non-zero elements of $\F$ are squares, and hence already non-primitive. It is easy to see that, for $q>3$, there exist consecutive non-squares. The deeper question is whether there exists a consecutive pair of elements that are both non-square and non-primitive.

From now on, assume $q$ is an odd prime power and define
\[
\theta_q := \frac{\phi(q-1)}{q-1}.
\] 
We call an element of $\F$ NSNP if it is both a non-square and non-primitive. An  NSNP pair is a pair of {\em consecutive} NSNP elements. 

If $\theta_q=1/2$, then every non-square is primitive; this occurs precisely when $q$ is a Fermat prime or $q=9$.  In these cases, no NSNP elements (and therefore no NSNP pairs) exist. Thus some restriction on $\theta_q$ is necessary.

For prime fields $\Fp$, the problem has been studied in terms of congruences modulo $p$. By elementary methods, Gun, Ramakrishnan, Sahu, and Thangadurai \cite{GRST2005} proved that if $\theta_p < 1/6$, then $\Fp$ contains an  NSNP pair. Using more elaborate sieving methods, Jarso and Trudgian \cite{JT2019} improved this to the weaker condition $\theta_p \le 1/4$. Previously, Gun, Luca, Rath, Sahu, and Thangadurai \cite{GLRST2007} had shown via character sum estimates, as a particular case of a more general result, that for sufficiently large primes $p$, an  NSNP pair exists whenever $\theta_p < 1/2 - \varepsilon$. When $\varepsilon = 1/6$ (so that $\theta_p < 1/3$), the argument in \cite{GLRST2007} appears to require a bound as large as $p > 10^{430}$ to guarantee the existence of an  NSNP pair.

\bigskip
Our main result is the following.

\begin{thm}\label{main}
Let $q$ be an odd prime power and define
\[
\theta_q := \frac{\phi(q-1)}{q-1}.
\]
If $\theta_q < \tfrac{1}{3}$, then $\F$ contains a pair of consecutive elements that are both non-square and non-primitive.

If $\theta_q = \tfrac{1}{3}$, then the same conclusion holds except when
\[
q \in \{7,13,19,25,37\}.
\]
\end{thm}

\medskip
To establish Theorem~\ref{main}, we study a more general problem. Let $\ell$ be an odd prime dividing $q-1$. We say that an element of $\F$ is an \emph{NS$\ell$ element} if it is both a non-square and an $\ell$th power, and define an \emph{NS$\ell$ pair} analogously. Clearly, every NS$\ell$ pair is an NSNP pair.

We prove the following general result, in which $\ell$ is taken to be the least odd prime divisor of $q-1$.

\begin{thm}\label{lpower}
Let $q$ be an odd prime power such that $\theta_q \le \tfrac{1}{3}$, and let $\ell$ be the least odd prime divisor of $q-1$. Then, except when $\ell = 3$ and
\[
q \in \{7,13,19,25,37,43\},
\]
the field $\F$ contains an NS$\ell$ pair.
\end{thm}

\begin{remark}
\textnormal{The exceptional values of $q$ in Theorem~\ref{main} correspond precisely to those for which $\theta_q = \tfrac{1}{3}$  (the bounary value) and no NSNP pair exists.}
\end{remark}

To interpret the hypotheses, observe that $\theta_q$ is the proportion of primitive elements in $\F^\times$. 
Since exactly half of $\F^\times$ consists of non-squares, the proportion of NSNP elements is
\[
\frac{1}{2} - \theta_q.
\]

Thus:
\begin{itemize}
\item $\theta_q < 1/6$ (as in \cite{GRST2005}) implies NSNP elements are at least twice as numerous as primitive elements;
\item $\theta_q \le 1/4$ (as in \cite{JT2019}) implies NSNP elements are at least as numerous as primitive elements;
\item $\theta_q \le 1/3$ (this paper) implies NSNP elements are at least half as numerous as primitive elements.
\end{itemize}

The existence problem for NSNP pairs is therefore comparable in difficulty to the existence of consecutive primitive pairs. However, primitive elements have a simpler multiplicative structure, and  (lower-bound) sieve methods apply more naturally to them. For NSNPs, appropriate  upper-bound sieves seem not to be available.

It is therefore natural to fix a prime $\ell \mid (q-1)$ and approach the problem of the  existence of NSNP pairs via that of the existence of NS$\ell$ pairs.
We first treat general $\ell$, and then specialise to the key small primes $\ell = 3$ and $\ell = 5$.

\begin{remark}
\textnormal{For later reference, we record which relevant values of $\theta_q$ occur for odd prime powers $q$:
\begin{itemize}
\item $\theta_q = 1/2$ if and only if $q$ is a Fermat prime or $q=9$;
\item $\theta_q = 1/3$ if and only if $q-1 = 2^a 3^b$ with $a,b \ge 1$;
\item $\theta_q \neq 1/4$ and $\theta_q \neq 1/6$, since neither $4$ nor $6$ is the Euler totient of any integer of the form $q-1$ with $q$ an odd prime power.
\end{itemize}
}
\end{remark}

 \bigskip
\section{A lower bound for the number of NS$\ell$ pairs}\label{lb}

Let $q$ be an odd prime power, and let $\ell$ be an odd prime dividing $q-1$. Since $\F^\times$ is cyclic of order $q-1$, there exists a multiplicative character $\eta$ of exact order $\ell$. Let $\lambda$ denote the quadratic character of $\F^\times$. Extend both $\lambda$ and $\eta$ to $\F$ by setting $\lambda(0)=0$ and $\eta(0)=0$. With this convention, all character sums over $\F$ are well-defined (see \cite{BEW1998}).

We seek a lower bound for
\[
N_{\ell}(q) = \#\{\alpha \in \F: \alpha, \alpha+1 \text{ are both non-squares and $\ell$th powers}\}.
\]

Define the characteristic function of elements of $\F$ that are both non-squares and $\ell$th powers as follows:
\[
f_\ell(\alpha) = \frac{1}{2\ell}\,(1-\lambda(\alpha)) \sum_{i=0}^{\ell-1} \eta^i(\alpha).
\]

It follows that
\[
N_\ell(q) = \sum_{\alpha \in \F\setminus \{0,-1\}} f_\ell(\alpha) f_\ell(\alpha+1).
\]
The exclusion of $\{0,-1\}$ is harmless, since $f_\ell(0)=0$ and $f_\ell(\alpha+1)=0$ when $\alpha=-1$.

Expanding $f_\ell(\alpha)f_\ell(\alpha+1)$ gives
\[
N_\ell(q) = \frac{1}{4\ell^2} \sum_{i,j=0}^{\ell-1} S(i,j),
\]
where
\begin{equation}\label{Neq2}
S(i,j) = J(\eta^i, \eta^j) - \lambda(-1) J(\lambda \eta^i, \eta^j) - J(\eta^i, \lambda \eta^j) + \lambda(-1) J(\lambda \eta^i, \lambda \eta^j).
\end{equation}
Here the Jacobi sum is defined by
\[
J(\chi_1,\chi_2) = \sum_{\alpha \in \F} \chi_1(\alpha)\chi_2(1-\alpha).
\]
This identity follows from the change of variables $\beta=-\alpha$, which transforms sums of the form $\sum \chi_1(\alpha)\chi_2(\alpha+1)$ into Jacobi sums, introducing factors of $\chi_1(-1)$ that account for the coefficients above.

We now evaluate or bound each contribution $S(i,j)$ using standard facts about Jacobi sums: if $\chi_1,\chi_2$ are nontrivial and $\chi_1\chi_2=1$, then $J(\chi_1,\chi_2)=-1$, while otherwise $|J(\chi_1,\chi_2)|=\sqrt{q}$.

\medskip

\noindent
\textbf{The block $S(0,0)$.}
Since $\eta^0=1$, we have
\[
S(0,0) = \sum_{\alpha \ne 0,-1} (1-\lambda(\alpha))(1-\lambda(\alpha+1)).
\]
Expanding,
\[
S(0,0)
= \sum_{\alpha \ne 0,-1} 1
- \sum_{\alpha \ne 0,-1} \lambda(\alpha)
- \sum_{\alpha \ne 0,-1} \lambda(\alpha+1)
+ \sum_{\alpha \ne 0,-1} \lambda(\alpha)\lambda(\alpha+1).
\]
These sums evaluate as follows:
\[
\sum_{\alpha \ne 0,-1} 1 = q-2,
\]
\[
\sum_{\alpha \ne 0,-1} \lambda(\alpha) = -\lambda(-1),
\quad
\sum_{\alpha \ne 0,-1} \lambda(\alpha+1) = -1,
\]
and
\[
\sum_{\alpha \ne 0,-1} \lambda(\alpha)\lambda(\alpha+1) = -1.
\]
Hence
\[
S(0,0) = q - 2 + \lambda(-1).
\]

\medskip

\noindent
\textbf{The blocks $S(0,j)$, $1 \le j \le \ell-1$.}
Using $J(1,\eta^j)=-1$ and $J(1,\lambda\eta^j)=-1$, while the remaining Jacobi sums have absolute value $\sqrt{q}$, we obtain
\[
|S(0,j)| \le 2\sqrt{q}.
\]
Summing over $j$ gives a total contribution bounded by
\[
2(\ell-1)\sqrt{q}.
\]

\medskip

\noindent
\textbf{The blocks $S(i,0)$, $1 \le i \le \ell-1$.}
We have
\[
J(\eta^i,1) = -1, \quad J(\lambda\eta^i,1) = -1,
\]
so that
\[
J(\eta^i,1) - \lambda(-1)J(\lambda\eta^i,1)
= \lambda(-1) - 1.
\]
Summing over $i$ yields a constant contribution
\[
(\ell-1)(\lambda(-1)-1).
\]
The remaining two Jacobi sums each have absolute value $\sqrt{q}$, giving an error term $E_1$ with
\[
|E_1| \le 2(\ell-1)\sqrt{q}.
\]

\medskip

\noindent
\textbf{The blocks with $1 \le i,j \le \ell-1$ and $i+j \equiv 0 \pmod{\ell}$.}
There are $\ell-1$ such pairs. In this case,
\[
J(\eta^i,\eta^{-i}) = -1, \quad
J(\lambda\eta^i,\lambda\eta^{-i}) = -\lambda(-1),
\]
while the remaining two Jacobi sums are bounded by $\sqrt{q}$. Thus the total contribution is
\[
-2(\ell-1) + E_2,
\quad
|E_2| \le 2(\ell-1)\sqrt{q}.
\]

\medskip

\noindent
\textbf{The remaining blocks.}
For the $(\ell-1)(\ell-2)$ pairs with $1 \le i,j \le \ell-1$ and $i+j \not\equiv 0 \pmod{\ell}$, none of the products of characters is trivial, so each Jacobi sum has absolute value $\sqrt{q}$. Hence the total contribution is bounded by
\[
4(\ell-1)(\ell-2)\sqrt{q}.
\]

\medskip

Collecting all contributions, we obtain
\[
\sum_{i,j} S(i,j)
\ge q - (\ell-2)(1+\lambda(-1)) - 3 - 4(\ell-1)^2\sqrt{q}.
\]
Substituting into  the expansion for $N_\ell(q)$ yields the following result.

\begin{thm}\label{thm:main}
Let $\ell$ be an odd prime dividing $q-1$. Then
\begin{equation}\label{mainbound}
N_\ell(q) \ge \frac{q - 4(\ell-1)^2 \sqrt{q} - (\ell-2)(1+\lambda(-1)) - 3}{4\ell^2}.
\end{equation}
\end{thm}

A weaker  (but simpler) criterion can be deduced from Theorem \ref{thm:main}.

\begin{cor}\label{Npos}
Suppose $\ell$ is an odd prime divisor of $q-1$ such that $q > 16 \ell^4$. Then $\F$ contains an  NS$\ell$ pair.
\end{cor}
\begin{proof}
Suppose \(q>16\ell^4\). Then \(\sqrt q>4\ell^2\), and hence
\[
\frac{(\ell-2)(1+\lambda(-1))+3}{\sqrt q}
<
\frac{2(\ell-2)+3}{4\ell^2}<1,
\]
since \(\ell\ge3\). Therefore, the numerator in \eqref{mainbound} is positive: indeed,
\[
\begin{aligned}
q-4(\ell-1)^2\sqrt q
&-(\ell-2)(1+\lambda(-1))-3\\
&>\sqrt q\bigl(\sqrt q-4(\ell-1)^2-1\bigr)>0,
\end{aligned}
\]
because \(\sqrt q>4\ell^2>4(\ell-1)^2+1\). Thus \(N_\ell(q)>0\), and \(\F\) contains an NS\(\ell\) pair.

\end{proof}

\medskip
  For illustration purposes, we append  Table \ref{smalltable} to demonstrate the application of Theorem \ref{thm:main} and Corollary \ref{Npos} for  the first few odd prime values of  $\ell$ (although in the sequel we only  require the first couple of rows).  
  
   In the table,  for each prime $\ell$ dividing $q-1$ listed, the first column  gives the minimum value of $q$ for which $N_\ell(q)$ is bound to be positive (as guaranteed by Theorem \ref{thm:main}). The second column gives the minimum value of $q$ which guarantees $N_\ell(q)$ by the weaker constraint of Corollary \ref{Npos}.    The benefit of the availability of Theorem \ref{thm:main} is particularly significant  when $\ell= 3$ or $5$.
   
   In column 3, it is assumed that $ \theta_q \le 1/3$ and $\ell$ is actually the least prime dividing $q-1$.  From these values, it is clear that, whenever $\ell \ge 7$, these conditions easily guarantee that $q >16\ell^4$ and so that $N_\ell(q)$ is positive.  All of this suggests that the next step should be a discussion of the situation when the least prime $\ell$ dividing $q-1$ is at least $7$.  This is accomplished in the next section.
  
\begin{table}[h]
\centering 
\caption{Bounds for small $\ell$}
\label{smalltable}
\begin{tabular}{c|c|c|c}   
$\ell$ 
& min $q$ (by Th \ref{thm:main})
& min $q$ (by Cor \ref{Npos}) 
& min $q$ with $\theta_q \le 1/3$  \\[2pt]
 \hline
3  & 269          & 1,296        & 31 \\
5  & 4,117        & 10,000       & 771 \\
7  & 20,769       & 38,416       & 646,647 \\
11 & 160,045      & 234,256      & 70,673,696,523\\
13 & 331,829      & 456,976       & {\small 28,215,721,625,272,767} \\
17 & 1,048,645    & 1,336,336     & {\small 37,158,896,445,334,596,135,027}\\
\end{tabular}
\end{table}

\bigskip

\section{The existence of NS$\ell$ pairs for $\ell \ge 7$}\label{bigL}

We begin by separating the case in which $\ell=7$.
\begin{thm}\label{l7}
Let $q$ be an odd prime power such the least odd prime divisor of $q-1$ is $7$ and that $\theta_q \le 1/3$. Then $\F$  contains an  $NS7$ pair.
\end{thm}

\begin{proof}
As hypothesised, suppose  that $q$ is such that $\ell=7 \mid (q-1)$ but $3 \nmid (q-1)$ and $5 \nmid (q-1)$, and also that $\theta_q \le 1/3$.   Let $\omega_q$ denote the number of distinct prime divisors of $q-1$.
If $\omega_q \le 5$, then
\[
\theta_q \ge \frac{1}{2} \cdot \frac{6}{7} \cdot \frac{10}{11} \cdot \frac{12}{13} \cdot \frac{16}{17} > 0.338.
\]
Hence $\omega_q \ge 6$ and
\[
\theta_q \ge \frac{1}{2} \cdot \frac{6}{7} \cdot \frac{10}{11} \cdot \frac{12}{13} \cdot \frac{16}{17} \cdot \frac{18}{19} > 0.3206.
\]
It follows that
\[
q \ge 2 \cdot 7 \cdot 11 \cdot 13 \cdot 17 \cdot 19 + 1 = 646,647> 16\ell^4,
\]
and the result holds by Corollary  \ref{Npos} (as in the third row of Table \ref{smalltable}).
\end{proof}

We have to devise an argument  that suffices for a general prime $\ell \geq 11$  even though a similar conclusion to that of Theorem \ref{l7} is ``obvious'' for any specific prime.

To achieve this, let the odd primes in increasing order be denoted $3=\ell_1, \ell_2, \dots$, so that $\ell = \ell_r$ for some $r \ge 3$. For $r \ge 1$, define the sum $S_r$, the product $P_r$, and the expression $\Phi_r$ by
\[
S_r = \sum_{i=r}^{2r} \frac{1}{\ell_i}, \qquad P_r = \prod_{i=r}^{2r} \ell_i, \qquad \Phi_r = \prod_{i=r}^{2r} \Big(1 - \frac{1}{\ell_i}\Big).
\]

\begin{lemma}\label{bertrand}
The sequence $S_r$ is decreasing.
\end{lemma}
\begin{proof}
\[
S_r -S_{r+1} =\frac{1}{\ell_r}-\frac{1}{\ell_{2r+1}}-\frac{1}{\ell_{2r+2}}>\frac{1}{\ell_r}-\frac{2}{\ell_{2r+1} }>\frac{1}{\ell_r}-\frac{2}{\ell_{r+1} }>0,
\]
since $\ell_{r+1} \le 2\ell_r$ by (the proof of) Bertrand's Postulate.
\end{proof}
\begin{lemma} \label{lbounds}
For $r \ge 4$ and $\ell = \ell_r \ge11 $, the following inequalities hold:
\[
2 P_r > 16 \ell^4 \qquad \text{and} \qquad \frac{1}{2} \Phi_r > \frac{1}{3}.
\]
\end{lemma}

\begin{proof}
Each factor in $P_r$ is at least $\ell$, so $P_r \ge \ell^{r+1} \ge \ell^5$ since $r \ge 4$. Hence $2 P_r \ge 2 \ell^5 > 16 \ell^4$.  

Next, using $\prod (1 - a_i) \ge 1 - \sum a_i$ for  $0\le a_i <1$, we have $\Phi_r \ge 1 - S_r \ge 1 - S_4$, by Lemma \ref{bertrand}. Since $\ell_8 = 19$, it follows that
\[
\frac{1}{2} \Phi_r \ge \frac{1}{2}(1 - S_4) > 0.36 > \frac{1}{3}.
\]
\end{proof}

\begin{thm}\label{largeL}
Let $q$ be an odd prime power such that $\theta_q \le 1/3$, and suppose $\ell \ge 11$ is the least odd prime divisor of $q-1$. Then $\F$ contains an  NS$\ell$ pair.
\end{thm}

\begin{proof}
Set $\ell = \ell_r$ for some $r \ge 4$, and let $k_r, \dots, k_s$ denote the odd prime factors of $q-1$ in increasing order.    Since the $i$th smallest odd prime divisor of $q-1$ is not less than the $i$th  smallest odd prime not less than $\ell$, i.e.,  $k_i \ge \ell_i$, we have
\[
\frac{1}{2} \prod_{i=r}^s \Big(1 - \frac{1}{\ell_i}\Big) \le \frac{1}{2} \prod_{i=r}^s \Big(1 - \frac{1}{k_i}\Big) = \theta_q \le \frac{1}{3}.
\]
By Lemma~\ref{lbounds}, it follows that $s > 2r$, so $q > 2 P_r > 16 \ell^4$. Applying Corollary~\ref{Npos} guarantees the existence of an NS$\ell$ pair.
\end{proof}
   
\bigskip

\section{The number of NS$\ell$ pairs as a cyclotomic number}

In Section~\ref{lb}, we derived lower bounds for $\mathrm{NS}_\ell(q)$ based primarily on the fact that Jacobi sums have absolute value $\sqrt{q}$ whenever neither of the characters involved, nor their product, is trivial. This resolves the problem for prime power values of $q$ when $\ell$, the least odd prime divisor of $q-1$, satisfies $\ell \geq 7$. 

It remains to examine a finite set of prime powers $q$ in the cases $\ell = 3$ or $5$. In these instances, it is natural to investigate whether the Jacobi sum method can be refined further, particularly since explicit evaluations are known for Jacobi sums of orders $6$ and $10$. See, for example, Chapters~3 and~11 of~\cite{BEW1998}, and in particular Theorems~11.6 and~11.7.

However, the author was unable to adapt the Jacobi sum approach effectively to exploit these explicit expressions in order to compute $N_3(q)$ or $N_5(q)$, nor even to obtain stronger lower bounds than those given in~\eqref{mainbound} that would demonstrate the positivity of these quantities in general.

There is, however, an alternative, closely related and natural approach, namely to take advantage of existing tables of cyclotomic numbers of orders 6 and 10.  Classically, such tables go back to Gauss and Jacobi  themselves, through Dickson \cite{Dic1935} and to Whiteman \cite{Whi1960}, Table 2,  and \cite{Whi1960a} as well as many others. 

Given {\em any}  odd prime $\ell$ dividing $q-1$  and a primitive element $\gamma$ of $\F$, the cyclotomic numbers of order $2\ell$ are defined as follows. For all pairs $(i,j), 0 \le i,j \le \ell-1$, the cyclotomic number $(i,j)_{2\ell}$ counts the number of $\alpha \in \F$  such that both $\alpha/\gamma^i$ and $(\alpha +1)/\gamma^j$ are non-zero $2\ell$th powers in $\F$.
      
  \begin{lemma}\label{cyclnum}
 Suppose $\ell$ is an odd prime divisor of $q-1$. Then $NS_\ell(q)= (\ell,\ell)_{2\ell}$.
   \end{lemma}
 \begin{proof}
 To be counted by $(\ell, \ell)_{2\ell}$, a pair $(\alpha,\alpha+1)$ must be such that $\alpha=\gamma^\ell\beta^{2\ell}$ for some $\beta \ne 0$.  Set $\beta =\gamma^s$. Then the index of $\alpha$ with respect to the generator $\gamma$ is $\ell(1+2s)$ which is odd (so $\alpha$ is not a square), yet divisible by $\ell$ (so $\alpha$ is an $\ell$th power).  The same is true for $\alpha +1$ and conversely. 
 \end{proof}   
 
 As we have indicated, tables for $(\ell,\ell)_{2\ell}$ for $\ell=3,5$  exist but can be elusive to locate. One problem is that although $ (\ell,\ell)_{2\ell}$ is an invariant of $\F$, the particular formula to apply depends on the choice of primitive element and the index of the element $2$ with respect to that generator. In principle, for a given $q$ it is less effort to  choose $\gamma$, find the index of $2$ and choose the appropriate formula from a table that can be used to search directly through  all the non-squares and $\ell$th powers in $\F$ but, in practice, with the convenience of computation nowadays, it is probably as reasonable to adopt the latter approach.

  \section{The existence of NS3 pairs}

Let $q$ be an odd prime power with $\ell=3\mid(q-1)$, so $q \equiv 1 \pmod{6}$.  In this case
the restriction $\theta_q\le1/3$ is automatically fulfilled.  Let $N_3(q)$ denote the
number of NS3 pairs in $\F$. 

From Theorem \ref{thm:main} with $\ell=3$,  $N_3(q) >0$ whenever $q\ge 269$ (see Table \ref{smalltable}), so it suffices to consider smaller values of $q$.
There are 29 such prime power values, namely those in the following set:
\begin{align*}
Q=\{\, & 7, 13, 19, 25,  31, 37,43,49, 61, 67,73, 79, 97, 103,  109, 121, 127, 139, 151, 157, 163,\\ & 169,181,  193,199, 211, 223,  229, 241 \,\}\,
\end{align*} 
 
 As it happens we can  show the existence of an NS3 pair in $\F$ quite simply   when $q$ is a Mersenne prime $q=2^m-1$, where $m$ is a prime exceeding 3.
 In particular,  this applies to the members 31, 127 of $Q$. 
\begin{lemma} \label{mers}
Suppose $q>7$ is a Mersenne prime .  Then $(q-2,q-1)=(-2,-1)$ is an $NS3$ pair.
\end{lemma} 
\begin{proof}
Write $q=2^m-1$, where $m >3$ is prime.  Evidently $q \equiv 3 \pmod {4}$  and so $-1$ is a non-square and consequently an NS3  element of $\F$. 

Hence -1 is an NS3 element of $\F$. Moreover, in $\F$, $2^m=1$, i.e., 2 has order $m$ and so index
\[
 \frac{q-1}{m}= \frac{2(2^{m-1}-1)}{m}
\]
 in $\F$.   Hence 2 is a square and therefore -2 is a non-square.     Also 3 divides $2^{m-1}-1$ and therefore 3 divides $(2^{m-1}-1)/m$ since $m>3$.
 Thus 2  is a cube and so -2 is a cube.  Hence $(-2,-1)$ is an NS3 pair.
\end{proof}
\bigskip
For all the simplicity of computation, it is instructive to use formulae to derive the values of $N_3(q)$ for $q \in Q$.  The particular reference we shall use is Theorem 2 of \cite{AK1995} (with its accompanying tables)  because its validity  extends to non-prime fields. Tables  1 and 2 (in \cite{AK1995}) list the various expressions for $36(3,3)$ that we need.  We use $t$ to denote the index of  $2$ with respect to the primitive element $\gamma$.   In particular,  Table 1 contains relevant expressions for the case in which $q \equiv 1 \pmod {12}$ and Table 2 for the case in which   $q \equiv 7 \pmod{12}$.  In Table 1 we focus on the entries for the number $(0,3) =(3,3)$ whereas in Table 2 we look at the expressions for $(0,0)=(3,3)$.  The entries in the tables are expressed in terms of the integers $A,B$, defined by $q=A^2+3B^2$, where the sign of $A$ agrees with the value $A\equiv 1 \pmod{3}$ and the sign of $B$ is such that $B \equiv -t  \pmod{3}$. 

In this way we obtain the precise value of $N_3(q)$ shown in Table \ref{N3qvalues} (of this paper).  The significance of this table is that it represents a  rare {\it explicit}  application of formulae for the evaluation of a cyclotomic number.  To explain the headings of the columns, we note that 
\begin{itemize}
\item $\gamma$ relates to the generating primitive element of $\F$, which for primes $q$ is always the least positive primitive root;
\item $t = \text{ind}_\gamma(2)$ is the index of the element 2 with respect to $\gamma$;
\item $A,B$ are such that $q=A^2+3B^2$ with sign conventions as described;
\item the heading $36(3,3)$ relates to the relevant expression in the tables used in the evaluation;
\item the final column denotes the actual  value of $N_3(q)=(3,3)$ obtained.  
\end{itemize}
\begin{remark} 
\textnormal {When $t$ is a multiple of 3, the sign of $B$ is undetermined which does not matter because in these cases the value of $B$ does not appear in the expression}
 \end{remark}
 \begin{remark}
\textnormal {The fact that the expression used to evaluate  $36(3,3)$ must be an integer divisible by 36  provides a valuable check on the validity of the method.}
 \end{remark}
 \begin{remark}
\textnormal {Strictly speaking, Theorem 2 of \cite{AK1995} assumes that, not only is $q-1$ a multiple of 3 but so also is $p-1$, where $q$ is a power of the prime $p$.   This is not the case when $q=25$ or $121$. Nevertheless, by lifting characters from the prime field to $\F$ by means of the  Davenport-Hasse theorem for Gauss and Jacobi sums  (see \cite{BEW1998}, Chapter 11), we justify  that the formulae are valid in these cases too.  Certainly, the values obtained this way turn out to be correct.}
 \end{remark}
 
  Since for the purpose of Theorem \ref{thm:main} we seek to describe precisely when $N_3(q)$ is positive, we record the following conclusion to be drawn from Table \ref{N3qvalues}
\begin{thm}\label{N3conclus}

Let $q$ be an odd prime power with $3\mid(q-1)$.  Then $\mathbb F_q$ contains
an $NS3$ pair if and only if $q$ is not a member of the set $\{7,13,19,25,37,43\}$.
\end{thm}

\bigskip
To confirm and extend the values of $N_3(q)$ through a moderate amount of computation, via Pari/GP, we  append a list of the actual NS3 pairs for each $q \in Q$.    These are displayed in Table \ref{N3qpairs} which is  in striking agreement with the predictions of Table 2.

\begin{table}[H]
\centering
\renewcommand{\arraystretch}{1.0}
\caption{Evaluating $N_3(q)$ using Tables 1 and 2 in \cite{AK1995}}
\label{N3qvalues}

\small
\begin{tabular}{ccccccc}
\toprule
$q$ & $\gamma$ & $t$ & $A$ & $B$ & $36(3,3)$ & $N_3(q)$ \\
\midrule
7   & 3 & 2   & -2  & 1      & $q-11-2A$     & 0 \\
13  & 2 & 1   & 1   & 2      & $q-5+4A-6B$   & 0 \\
19  & 2 & 1   & 4   & -1     & $q-11-2A$     & 0 \\
25  & $2+\sqrt{2}$ & 6 & -5 & 0 & $q-5+4A$    & 0 \\
31  & 3 & 24  & -2  & $\pm 3$ & $q-11-8A$    & 1 \\
37  & 2 & 1   & -5  & 2      & $q-5+4A+6B$   & 0 \\
43  & 3 & 7   & 4   & $\pm 3$ & $q-11-8A$    & 0 \\
49  & $1+\sqrt{3}$ & 40 & 7 & 0 & $q-5+4A-6B$ & 2 \\
61  & 2 & 1   & 7   & 2      & $q-5+4A-6B$   & 2 \\
67  & 2 & 1   & -8  & -1     & $q-11-2A$     & 2 \\
73  & 5 & 34  & -5  & 4      & $q-5+4A+6B$   & 2 \\
79  & 3 & 4   & -2  & 5      & $q-11-2A$     & 2 \\
97  & 5 & 34  & -7  & -4     & $q-5+4A-6B$   & 4 \\
103 & 5 & 44  & 10  & 1      & $q-11-2A$     & 2 \\
109 & 6 & 57  & 1   & $\pm 6$ & $q-5+4A$     & 3 \\
121 & $1+\sqrt{2}$ & 12 & -11 & 0 & $q-5+4A$  & 2 \\
127 & 3 & 72  & 10  & $\pm 3$ & $q-11-8A$    & 1 \\
139 & 2 & 1   & -8  & 5      & $q-11-2A$     & 3 \\
151 & 6 & 98  & -2  & 7      & $q-11-2A$     & 4 \\
157 & 5 & 141 & 7   & $\pm 6$ & $q-5+4A$     & 5 \\
163 & 2 & 1   & 4   & 7      & $q-11-2A$     & 4 \\
169 & $1+2\sqrt{2}$ & 70 & -11 & -4 & $q-5+4A-6B$ & 4 \\
181 & 2 & 1   & 13  & 2      & $q-5+4A-6B$   & 6 \\
193 & 5 & 34  & 1   & 8      & $q-5+4A-6B$   & 4 \\
199 & 2 & 1   & -14 & -1     & $q-11-2A$     & 6 \\
211 & 2 & 1   & -8  & -7     & $q-11-2A$     & 6 \\
223 & 3 & 180 & -14 & $\pm 3$ & $q-11-8A$    & 9 \\
229 & 6 & 21  & -11 & $\pm 6$ & $q-5+4A$     & 5 \\
241 & 7 & 190 & 7   & 8      & $q-5+4A-6B$   & 6 \\
\bottomrule
\end{tabular}
\end{table}

\bigskip

\begin{table}[H]
\centering
\renewcommand{\arraystretch}{1.0}
\caption{Consecutive non-square cubic pairs}
\label{N3qpairs}

\small
\begin{tabular}{c c p{0.65\textwidth}}
\toprule
$q$ & $N_3(q)$ & NS3 pairs \\
\midrule

31  & 1 & $(29,30)$ \\
49  & 2 & $(3+\sqrt3,4+\sqrt3),\ (3+6\sqrt3,4+6\sqrt3)$ \\
61  & 2 & $(23,24),\ (37,38)$ \\
67  & 2 & $(42,43),\ (52,53)$ \\
73  & 2 & $(21,22),\ (51,52)$ \\
79  & 2 & $(14,15),\ (57,58)$ \\
97  & 4 & $(19,20),\ (45,46),\ (51,52),\ (77,78)$ \\
103 & 2 & $(89,90),\ (94,95)$ \\
109 & 3 & $(32,33),\ (54,55),\ (76,77)$ \\
121 & 2 & $(5+2\sqrt2,6+2\sqrt2),\ (5+9\sqrt2,6+9\sqrt2)$ \\
127 & 1 & $(125,126)$ \\
139 & 4 & $(59,60),\ (74,75),\ (75,76),\ (94,95)$ \\
151 & 4 & $(26,27),\ (27,28),\ (131,132),\ (142,143)$ \\
157 & 5 & $(7,8),\ (28,29),\ (78,79),\ (128,129),\ (149,150)$ \\
163 & 4 & $(27,28),\ (30,31),\ (98,99),\ (141,142)$ \\
169 & 4 & $(6+\sqrt2,7+\sqrt2),\ (6+3\sqrt2,7+3\sqrt2), \ (6+10\sqrt2,7+10\sqrt2),\ (6+12\sqrt2,7+12\sqrt2)$ \\
181 & 6 & $(6,7),\ (7,8),\ (30,31),\ (150,151),\ (173,174),\ (174,175)$ \\
193 & 4 & $(87,88),\ (88,89),\ (104,105),\ (105,106)$ \\
199 & 6 & $(11,12),\ (59,60),\ (82,83),\ (135,136),\ (136,137),\ (137,138)$ \\
211 & 6 & $(27,28),\ (67,68),\ (88,89),\ (89,90),\ (97,98),\ (146,147)$ \\
223 & 9 & $(26,27),\ (103,104),\ (189,190),\ (190,191),\ (206,207),$\\
    &&$(207,208), \ (208,209),\ (215,216),\ (221,222)$ \\
229 & 5 & $(21,22),\ (106,107),\ (114,115),\ (122,123),\ (207,208)$ \\
241 & 6 & $(43,44),\ (101,102),\ (102,103),\ (138,139),\ (139,140),\ (197,198)$ \\
\bottomrule

\end{tabular}
\end{table}

 \FloatBarrier
 \bigskip 
 \section{The existence of NS5 pairs}

Let $q$ be an odd prime power such that the least odd prime divisor of $q-1$ is $5$.  

We are assuming $3\nmid (q-1)$.  The first possibility is that $3|q$, i.e., $q$ is a power of $3$, in which case $q=3^{4s}$, say, because $5|(q-1)$.  The other possibility is that $3|(q+1)$ in which case $q \equiv 11 \pmod{30}$.  From Section \ref{lb}, if $q>4117$ and $\theta_q \le 1/3$, then $\F$ contains an NS5 pair.   The only prime power $q=3^{4s}< 4117$ is $81$ and $\theta_{81}  =2/5>1/3$,  So, as far as existence only is concerned, it suffices to restrict our attention to smaller prime powers $q\equiv 11 \pmod{30}$. Consequently, after calculation it remains only to consider  the five  prime powers in the set $\{911,\ 1331=11^3,\ 2381,\ 3221,\ 3851\}$.

Attempts to improve the lower bound of Theorem \ref{thm:main}  or find exact values for $N_5(q)$ by means of explicit values for Jacobi sums or Whiteman's tables (\cite{Whi1960a}) for $(5,5)_{10}$ did not justify the effort. Hence we complete this section with a table of values for $N_5(q)$ obtained through Pari/GP. This computation also delivered a full list of the actual $NS5$ pairs.

  \begin{table}[h]
\centering
\begin{tabular}{c|c}
\hline
$q$ & $N_5(q)$ \\
\hline
911  & 15 \\
1331 & 12 \\
2381 & 25 \\
3221 & 35 \\
3851 & 38 \\
\hline
\end{tabular}
\end{table}

\bigskip
\section{Proof of Theorem \ref{main} and some further questions}
For all the exceptional prime powers $q$ listed in Theorem \ref{lpower} except $q=43$, an NS3 pair is the same as an NSNP pair and therefore these prime powers remain exceptional in Theorem \ref{main}.   When $q=43$  it suffices to verify the following lemma.  To see this, we observe that $(7,8)$ is an NSNP pair since $7$ is a non-square seventh power and $8$ is a non-square cube. Hence $q=43$ is not exceptional for Theorem \ref{main}.

\begin{lemma} \label{case43}
In $\mathbb{F}_{43}$, $7$ is a non-square $7$th power and $8$ is a non-square cube.  Hence $(7,8)$ is a NSNP pair.
\end{lemma}

\begin{proof}
7 is a non-square by quadratic reciprocity,  Also $7^3\equiv 7\cdot 49  \equiv 7\cdot 6 \equiv -1 \pmod{43}$ so that 7 has order $6=(q+1)/7$. Hence $7$ is a $7$th power. 
The element $8$ is a cube, and $2$ is a non-square.
\end{proof}
\begin{remark}
\textnormal{In fact, $(7,8)$ is the sole NSNP pair in $\mathbb{F}_{43}$.}
\end{remark}

\bigskip
It remains to suggest a few questions for further investigation.

\begin{question}\label{Q1}

\textnormal{In this paper we have throughout considered only odd prime powers in which situation it is natural to partition the set of non-primitive elements of  $\F$ into the set of non-zero squares and the set of non-primitive non-squares.  If $q$ is even, i.e., a power of $2$, then (more simply)  every non-zero element is either primitive or non-primitive. 
At one extreme, when $q-1$ is prime, all non-zero elements except $1$ are primitive.   When $q$ is even,  is there a condition on $\theta_q$  (comparable to the restriction $\theta_q \le 1/3$) that would guarantee that $\F$ contains a NP pair?}
\end{question}

\medskip
\begin{question}\label{Q2}
\textnormal{Suppose $q$ is an odd prime power and $\ell$ is the least odd prime divisor of $q-1$.  It seems, from Section \ref{bigL}, that the conditions that $\theta_q \le 1/3$ might be more restrictive than necessary when $\ell \ge7$.  Is there a condition, for example, $\displaystyle{\theta_q<\frac{\ell-1}{2\ell}}$, which would guarantee the existence of a NS$\ell$ pair? }
\end{question}

\medskip

\begin{question}\label{Q3}
\textnormal{Suppose $q$ is  a power of a prime $p\ge3$. From \cite{GLRST2007} we know (at least when $q$ itself  is prime)  that, for sufficiently large $q$,  there exists  
three consecutive NSNP elements in $\F$ that are each non-square and non-primitive.   Is there a condition of the form  $\theta_q<c$ that would guarantee this?
}

\end{question}

\bigskip
 We record finally that, following the appearance of  a preliminary version of this article on arXiv Mathematics, these questions have already been tackled by a number of  authors.  First, I am grateful to  Kaimin Cheng  (China West Normal University) for forwarding his  note entitled    
 {\it Consecutive non-square non-primitive elements in finite fields}  which relates to each of them.  Further, two related items have also appeared on arXiv.  Specifically, 
 \cite{SB2026} answers Question \ref{Q1} and \cite{ZW2026} deals with Questions \ref{Q2} and \ref{Q3} and more.
  
\bigskip
\subsection*{Acknowledgment}

The author was formerly Professor of Number Theory at the University of Glasgow.  He is grateful to the university for continuing use of library facilities.

\bigskip
\noindent
Stephen D. Cohen\\
6 Bracken Road, Portlethen, \\
Aberdeen, AB12 4TA, Scotland, UK\\
E-mail: stephen.cohen@glasgow.ac.uk
\end{document}